\documentclass[a4paper,12pt]{article}
\usepackage[utf8]{inputenc}
\usepackage{amsmath,amssymb, amsthm,authblk}
\usepackage{latexsym}
\usepackage{graphicx}
\usepackage{amsrefs}
\usepackage{float}
\usepackage{tikz}
\usetikzlibrary{calc}
\usepackage{color}

\newtheorem{theorem}{Theorem}
\newtheorem{defi}{Definition}

\newtheorem{corollary}[theorem]{Corollary}
\newtheorem{proposition}[theorem]{Proposition}

\newtheorem{remark}{Remark}

\begin{document}

\title{Decomposition of a graph into two disjoint odd subgraphs}

\author[1]{Mikio Kano\thanks{mikio.kano.math@vc.ibaraki.ac.jp}}
\author[2,3]{Gyula Y Katona\thanks{kiskat@cs.bme.hu}}
\author[2]{Kitti Varga\thanks{vkitti@cs.bme.hu}}
\affil[1]{Ibaraki University, Hitachi, Ibaraki, Japan}
\affil[2]{Department of Computer Science and
Information Theory, Budapest University of Technology and Economics, Hungary}
\affil[3]{MTA-ELTE Numerical Analysis and Large Networks Research Group, Hungary}
\date{\today}

\maketitle

\begin{abstract}
 An odd (resp.~even) subgraph in a multigraph is its subgraph in which every vertex has odd (resp.~even) degree. We say that a multigraph can be decomposed into two odd subgraphs if its edge set can be partitioned into two sets so that both form odd subgraphs. In this paper we give a necessary and sufficient condition for the decomposability of a multigraph into two odd subgraphs. We also present a polynomial time algorithm for finding such a decomposition or showing its non-existence. We also deal with the case of the decomposability into an even subgraph and an odd subgraph.
\end{abstract}

\section{Introduction}
In this paper we mainly consider {\em multigraphs}, which may have multiple edges but have no loops.
A graph without multiple edges or loops is called a {\em simple graph}.
Let $G$ be a  multigraph with vertex set $V(G)$ and edge set $E(G)$. The number of vertices in $G$ is called its {\em order} and denoted by $|G|$, and the number of edges in $G$ is called its {\em size}
and denoted by $e(G)$. Let $EvenV(G)$ denote the set of vertices of even degree and $OddV(G)$ denote the set of vertices of odd degree.
For a vertex set $U$ of $G$, the subgraph of $G$ induced by $U$ is denoted by $\langle U \rangle _G$.  For two disjoint vertex sets $U_1$ and $U_2$ of $G$, the number of edges between $U_1$ and $U_2$ is denoted by $e_G(U_1, U_2)$.
For a vertex $v$ of $G$, the degree of $v$ in $G$ is denoted by $\deg_G(v)$.
Moreover, when some edges of $G$ are colored with red and blue, 
for a vertex $v$, the number of red edges incident with $v$
is denoted by $\deg_{\textrm{red}}(v)$, and the number of red edges in $G$ is denoted by $e_{\textrm{red}}(G)$. Analogously, $\deg_{\textrm{blue}}(v)$ and $e_{\textrm{blue}}(G)$ are defined.

An odd (resp.~even) subgraph of $G$ is a subgraph in which every vertex has odd (resp.~even) degree. An odd factor of $G$ is a spanning odd subgraph of $G$. It is obvious by the handshaking lemma that every connected multigraph containing an odd factor has even order. This condition is also sufficient as follows.

\begin{proposition}[Problem $42$ of \S 7 in \cite{lovasz}] \label{oddfactor}
 A multigraph $G$ has an odd factor if and only if every component of $G$ has even order.
\end{proposition}

Moreover, such an odd factor, if it exists, can be found in polynomial time (Problem $42$ of \S 7 in \cite{lovasz}). Consider a connected multigraph of even order on the vertices $v_1, \ldots, v_{2m}$ and for any $i \in \{1,2, \ldots,m\}$, fix a path $P_i$ connecting $v_i$ and $v_{i+m}$. Then the edges appearing odd times in the paths $P_1, \ldots, P_m$ forms an odd factor. The above proposition also follows from the fact that for a tree $T$ of even order, the set
\[ \{e\in E(T): \mbox{$T-e$ consists of two odd components}\} \]
forms an odd factor of $T$.

We say that $G$ can be decomposed into $n$ odd subgraphs if its edge set can be partitioned into $n$ sets $E_1, \ldots, E_n$ so that for every $i \in \{ 1, \ldots, n \}$, $E_i$ forms an odd subgraph. Some authors say that in this case $G$ is odd $n$-edge-colorable.

Our main result gives a criterion for a  multigraph to be decomposed into two odd subgraphs, and proposes a polynomial time algorithm for finding such a decomposition or showing its
non-existence.

We begin with some known results related to ours.
 
\begin{theorem}[\cite{pyber}] \label{pyber-1}
Every simple graph can be decomposed into four odd subgraphs.
\end{theorem}

This upper bound is sharp, for example, the wheel of four spokes ($W_4$, see Figure~1) cannot be decomposed into three odd subgraphs. In \cite{matrai} M\'atrai constructed an infinite family of graphs with the same property.

\begin{theorem}[\cite{pyber}] \label{pyber-2}
Every forest can be decomposed into two odd subgraphs.
\end{theorem}

\begin{theorem}[\cite{pyber}] \label{pyber-3}
Every connected simple graph of even order can be decomposed into three odd subgraphs.
\end{theorem}

Since every connected simple graph $G$ of even order has an odd factor,
if we take an odd factor $F$ with maximum size, then $G-E(F)$ becomes a forest, and it can be decomposed into two odd subgraphs by Theorem~\ref{pyber-2}. Thus Theorem~\ref{pyber-3} follows.

\begin{theorem}[\cite{luzar}] \label{luzar}
Every connected multigraph can be decomposed into six odd subgraphs. And equality holds if and only if the multigraph is a Shannon triangle of type $(2,2,2)$  (see Figure~1).
\end{theorem}

\begin{theorem}[\cite{petrusevski}] \label{petrusevski}
 Every connected  multigraph can be decomposed into four odd subgraphs except for the Shannon triangles of type $(2,2,2)$ and $(2,2,1)$ (see Figure~1).
\end{theorem}

\begin{figure}[H]
 \begin{center}
 \begin{tikzpicture}
  \tikzstyle{vertex}=[draw,circle,fill=black,minimum size=3,inner sep=0]
  
  \node[vertex] (w) at (0,0) {};
  \node[vertex] (v1) at (0:1) {};
  \node[vertex] (v2) at (90:1) {};
  \node[vertex] (v3) at (180:1) {};
  \node[vertex] (v4) at (270:1) {};
 
  \draw (w) -- (v1);
  \draw (w) -- (v2);
  \draw (w) -- (v3);
  \draw (w) -- (v4);
  \draw (v1) -- (v2) -- (v3) -- (v4) -- (v1);
  
  \node at (-1,1) {$W_4$};
  
 \begin{scope}[shift={(3,0)}]
 
  \node[vertex] (a1) at (90:1.25) {};
  \node[vertex] (a2) at (210:1.25) {};
  \node[vertex] (a3) at (330:1.25) {};
  
  \draw (a1) to [bend right=30] (a2);
  \draw (a2) to [bend right=30] (a3);
  \draw (a3) to [bend right=30] (a1);
  \draw (a1) to [bend left=30] (a2);
  \draw (a2) to [bend left=30] (a3);
  \draw (a3) to [bend left=30] (a1);
  \draw (a1) to node[pos=0.5, fill=white, rotate=-30] {odd} (a2);
  \draw (a2) to node[pos=0.5, fill=white, rotate=-90] {odd} (a3);
  \draw (a3) to node[pos=0.5, fill=white, rotate=30] {odd} (a1);
 \end{scope}
 
 \begin{scope}[shift={(6,0)}]
 
  \node[vertex] (a1) at (90:1.25) {};
  \node[vertex] (a2) at (210:1.25) {};
  \node[vertex] (a3) at (330:1.25) {};
  
  \draw (a1) to [bend right=30] (a2);
  \draw (a2) to [bend right=30] (a3);
  \draw (a3) to [bend right=30] (a1);
  \draw (a1) to [bend left=30] (a2);
  \draw (a2) to [bend left=30] (a3);
  \draw (a3) to [bend left=30] (a1);
  \draw (a1) to node[pos=0.5, fill=white, rotate=-30] {even} (a2);
  \draw (a2) to node[pos=0.5, fill=white, rotate=-90] {odd} (a3);
  \draw (a3) to node[pos=0.5, fill=white, rotate=30] {odd} (a1);
 \end{scope}
 
 \begin{scope}[shift={(1.5,-3)}]
 
  \node[vertex] (a1) at (90:1.25) {};
  \node[vertex] (a2) at (210:1.25) {};
  \node[vertex] (a3) at (330:1.25) {};
  
  \draw (a1) to [bend right=30] (a2);
  \draw (a2) to [bend right=30] (a3);
  \draw (a3) to [bend right=30] (a1);
  \draw (a1) to [bend left=30] (a2);
  \draw (a2) to [bend left=30] (a3);
  \draw (a3) to [bend left=30] (a1);
  \draw (a1) to node[pos=0.5, fill=white, rotate=-30] {even} (a2);
  \draw (a2) to node[pos=0.5, fill=white, rotate=-90] {even} (a3);
  \draw (a3) to node[pos=0.5, fill=white, rotate=30] {odd} (a1);
 \end{scope}
 
 \begin{scope}[shift={(4.5,-3)}]
 
  \node[vertex] (a1) at (90:1.25) {};
  \node[vertex] (a2) at (210:1.25) {};
  \node[vertex] (a3) at (330:1.25) {};
  
  \draw (a1) to [bend right=30] (a2);
  \draw (a2) to [bend right=30] (a3);
  \draw (a3) to [bend right=30] (a1);
  \draw (a1) to [bend left=30] (a2);
  \draw (a2) to [bend left=30] (a3);
  \draw (a3) to [bend left=30] (a1);
  \draw (a1) to node[pos=0.5, fill=white, rotate=-30] {even} (a2);
  \draw (a2) to node[pos=0.5, fill=white, rotate=-90] {even} (a3);
  \draw (a3) to node[pos=0.5, fill=white, rotate=30] {even} (a1);
 \end{scope}
 \end{tikzpicture}
 \caption{The wheel $W_4$ and the Shannon triangles of type (1,1,1), (2,1,1), (2,2,1), (2,2,2).} \label{shannon}
 \end{center}
\end{figure}
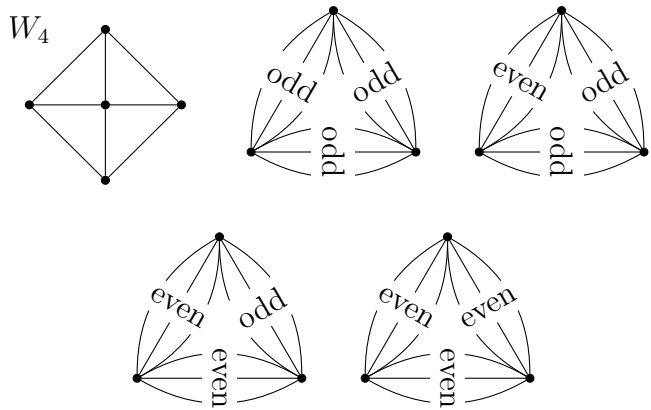

We say that $G$ can be covered by $n$ odd subgraphs if its edge set can be covered by $n$ sets $E_1, \ldots, E_n$ (not necessarily  disjointly)  so that for every $i \in \{ 1, \ldots, n \}$, $E_i$ forms an odd subgraph.

\begin{theorem}[\cite{matrai}]
Every connected  multigraph of odd order can be covered by three odd subgraphs.
\end{theorem}

In this paper we study the decomposability of a  multigraph into an even subgraph and an odd subgraph, and into two odd subgraphs. We also remark that the case of decomposing into two even subgraphs is trivial.

\begin{proposition} \label{even+odd}
 A multigraph $G$ can be decomposed into an even subgraph and an odd subgraph if and only if every component of $\langle OddV(G) \rangle_G$ has even order.
\end{proposition}

\begin{proof}
 Such a decomposition exists if and only if there is an odd factor in $\langle OddV(G) \rangle_G$, since all edges incident with any vertex of even degree must belong to the even subgraph. So by Proposition \ref{oddfactor}, the proposition follows.
\end{proof}

Since an odd factor can be found in polynomial time, we can conclude the following.

\begin{corollary} \label{even+odd_poly}
 There is a polynomial time algorithm for decomposing a multigraph into an odd subgraph and an even subgraph or showing the non-exis\-tence of such a decomposition.
\end{corollary}

\begin{remark}
 The case of the decomposability into two even subgraphs is trivial: a multigraph can be decomposed into two even subgraphs if and only if every vertex of the multigraph has even degree.
\end{remark}

The following two theorems are our main results.

\begin{theorem} \label{n&s_condition}
 Let $G$ be a multigraph and let $\mathcal{X}$ denote the set of  components of $\langle OddV(G) \rangle_G$, and let $\mathcal{Y}$ and $\mathcal{Z}$ denote the sets of components of $\langle EvenV(G) \rangle_G$ with odd order and even order, respectively.
 Now $G$ can be decomposed into two odd subgraphs if and only if for every $\mathcal{S} \subseteq \mathcal{Y} \cup \mathcal{Z}$ with $|\mathcal{S} \cap \mathcal{Y}|$  odd, there exists $X \in \mathcal{X}$ that has neighbors in odd number of components of $\mathcal{S}$.
\end{theorem}

\begin{theorem} \label{2odd_poly} 
 There is a polynomial time algorithm for decomposing a multigraph into two odd subgraphs or showing the non-existence of such a decomposition.
\end{theorem}

\section{Proofs of Theorems~\ref{n&s_condition} and \ref{2odd_poly} }

We begin with a definition and a proposition on it.

\begin{defi}
 Let $G$ be a multigraph and $T \subseteq V(G)$. A subgraph $J$ of $G$ is called a $T$-join if $OddV(J)=T$.
\end{defi}

\begin{proposition}[\cite{Edmonds1973}] \label{Tjoin}
 Let $G$ be a multigraph and $T \subseteq V(G)$. There exists a $T$-join in $G$ if and only if every component of $G$ contains an even number of vertices of $T$.
\end{proposition}

The following theorem gives another necessary and sufficient condition for a multigraph to be decomposed into two odd subgraphs.

\begin{theorem} \label{partition-thm}
 Let $G$ be a multigraph and $\mathcal{Y}$ and $\mathcal{Z}$ denote the sets of components of $\langle EvenV(G) \rangle_G$ with odd order and even order, respectively. Then $G$ can be decomposed into two odd subgraphs if and only if there exists a partition $\mathcal{R} \cup \mathcal{B}$ of the components of $\langle OddV(G) \rangle_G$ such that
 \begin{enumerate}
  \item[(i)] $e_G(R,Y)$ and $e_G(B,Y)$ are both odd for every $Y \in \mathcal{Y}$, and
  \item[(ii)] $e_G(R,Z)$ and $e_G(B,Z)$ are both even for every $Z \in \mathcal{Z}$,
 \end{enumerate}
 where $R$ and $B$ are the sets of vertices that belong to the components in $\mathcal{R}$ and $\mathcal{B}$, respectively. 

\end{theorem}

\begin{figure}[H]
\begin{center}
\begin{tikzpicture}[scale=1.5]
\tikzstyle{vertex}=[draw,circle,fill=black,minimum size=3,inner sep=0]

\draw (-1.75,1.5) ellipse (0.4 and 0.3);
\draw (-0.75,1.5) ellipse (0.4 and 0.3);
\draw (0.75,1.5) ellipse (0.4 and 0.3);
\draw (1.75,1.5) ellipse (0.4 and 0.3);

\node at (3.5,1.5) {$OddV(G)$};

\draw (-1.75,0) ellipse (0.4 and 0.3);
\draw (-0.75,0) ellipse (0.4 and 0.3);
\draw (0.75,0) ellipse (0.4 and 0.3);
\draw (1.75,0) ellipse (0.4 and 0.3);

\node at (-1.75,0) {$Y$};
\node at (1.75,0) {$Z$};

\node at (3.5,0) {$EvenV(G)$};

\node at (-1.25,2) {$\mathcal{R}$};
\node at (1.25,2) {$\mathcal{B}$};
\node at (-1.25,-0.5) {$\mathcal{Y}$};
\node at (1.25,-0.5) {$\mathcal{Z}$};

\draw ($(-1.75,0)+(90+39:0.45 and 0.35)$) -- ($(-1.75,1.5)+(270-45:0.45 and 0.35)$);
\draw ($(-1.75,0)+(90+32:0.45 and 0.35)$) -- ($(-1.75,1.5)+(270-35:0.45 and 0.35)$);
\draw ($(-1.75,0)+(90+25:0.45 and 0.35)$) -- ($(-1.75,1.5)+(270-25:0.45 and 0.35)$);
\draw ($(-1.75,0)+(90+20:0.45 and 0.35)$) -- ($(-0.75,1.5)+(225-5:0.45 and 0.35)$);
\draw ($(-1.75,0)+(90+10:0.45 and 0.35)$) -- ($(-0.75,1.5)+(225+5:0.45 and 0.35)$);
\draw ($(-1.75,0)+(90-0:0.45 and 0.35)$) -- ($(-0.75,1.5)+(225+15:0.45 and 0.35)$);
\node[fill=white] at (-1.85,0.65) {odd};
\draw ($(-1.75,0)+(90-5:0.45 and 0.35)$) -- ($(0.75,1.5)+(225-5:0.45 and 0.35)$);
\draw ($(-1.75,0)+(90-15:0.45 and 0.35)$) -- ($(0.75,1.5)+(225+5:0.45 and 0.35)$);
\draw ($(-1.75,0)+(90-25:0.45 and 0.35)$) -- ($(0.75,1.5)+(225+15:0.45 and 0.35)$);
\draw ($(-1.75,0)+(90-35:0.45 and 0.35)$) -- ($(1.75,1.5)+(225-10:0.45 and 0.35)$);
\draw ($(-1.75,0)+(90-45:0.45 and 0.35)$) -- ($(1.75,1.5)+(225:0.45 and 0.35)$);
\draw ($(-1.75,0)+(90-55:0.45 and 0.35)$) -- ($(1.75,1.5)+(225+10:0.45 and 0.35)$);
\node[fill=white] at (-0.85,0.65) {odd};

\draw ($(1.75,0)+(90+35:0.45 and 0.35)$) -- ($(-1.75,1.5)+(315-5:0.45 and 0.35)$);
\draw ($(1.75,0)+(90+25:0.45 and 0.35)$) -- ($(-1.75,1.5)+(315+5:0.45 and 0.35)$);
\draw ($(1.75,0)+(90+15:0.45 and 0.35)$) -- ($(-0.75,1.5)+(315-5:0.45 and 0.35)$);
\draw ($(1.75,0)+(90+5:0.45 and 0.35)$) -- ($(-0.75,1.5)+(315+5:0.45 and 0.35)$);
\node[fill=white] at (0.9,0.65) {even};
\draw ($(1.75,0)+(90-5:0.45 and 0.35)$) -- ($(0.75,1.5)+(315-5:0.45 and 0.35)$);
\draw ($(1.75,0)+(90-15:0.45 and 0.35)$) -- ($(0.75,1.5)+(315+5:0.45 and 0.35)$);
\draw ($(1.75,0)+(90-25:0.45 and 0.35)$) -- ($(1.75,1.5)+(270+25:0.45 and 0.35)$);
\draw ($(1.75,0)+(90-32:0.45 and 0.35)$) -- ($(1.75,1.5)+(270+35:0.45 and 0.35)$);
\node[fill=white] at (1.8,0.65) {even};
  
\end{tikzpicture}
\caption{The structure of the multigraph.}
\label{}
\end{center}
\end{figure}
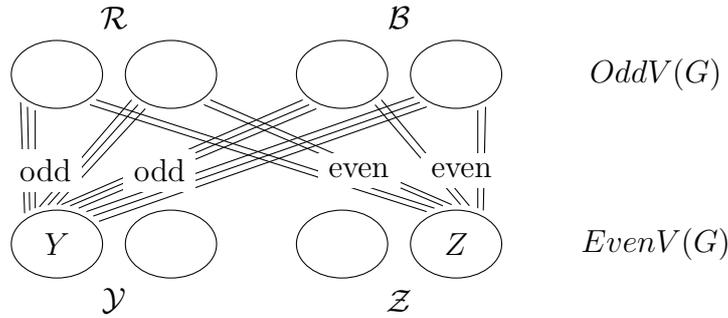

\begin{proof}
 Suppose that $G$ can be decomposed into two odd subgraphs, and color the edges of one with red and with blue the other. Obviously, if a vertex of $G$ has odd degree, then all edges incident with it must have the same color. Consider an arbitrary component $X$ of $\langle OddV(G) \rangle_G$. Then all edges that have at least one endpoint in $X$ have the same color. Let $\mathcal{R}$ and $\mathcal{B}$ denote the set of those components of $\langle OddV(G) \rangle_G$ in which the edges are red and blue, respectively. Let $Y \in \mathcal{Y}$. Then
 \[ \sum_{v \in Y} \deg_{\textrm{red}}(v) = 2 e_{\textrm{red}}(\langle Y\rangle_G) + e_G(R,Y) \text{.} \]
 Since $|Y|$ is odd and $\deg_{\textrm{red}}(v)$ is odd for every $v \in Y$, the left side of the equation is odd, and so $e_G(R,Y)$ must be odd. Similarly, $e_G(B,Y)$ is also odd, and $e_G(R,Z)$ and $e_G(B,Z)$ are both even. Therefore, the conditions (i) and (ii) hold.
 
 \medskip
 
 Next assume that there exists a partition $\mathcal{R} \cup \mathcal{B}$ satisfying (i) and (ii). Then color all the edges incident with any vertex of $R$ red and all the edges incident with any vertex of $B$ blue. Note that no edge of $\langle EvenV(G)\rangle_G$ is colored now, and there exist no edges between $R$ and $B$. Let $T \subseteq EvenV(G)$ be the set of vertices having even red-degree in this stage.
 
 Now we show that every component of $\langle EvenV(G) \rangle_G$ contains an even number of vertices of $T$. Let $Y \in \mathcal{Y}$. Then by condition (i),
 \[ \sum_{v \in Y}\deg_{\textrm{\textrm{red}}}(v) = \underbrace{\sum_{v \in Y \cap T} \underbrace{\deg_{\textrm{red}}(v)}_{\text{even}}}_{\text{even}} + \sum_{v \in Y\setminus T} \underbrace{\deg_{\textrm{red}}(v)}_{\text{odd}} = \underbrace{e_G(R,Y)}_{\text{odd}} \text{.} \]
Hence $|Y\setminus T|$ is odd, and since $|Y|$ is odd, $|Y \cap T|$ must be even. 
By the same argument given above, for any $Z \in \mathcal{Z}$, it follows that $|Z|$ is even and $e_G(R,Z)$ is even by the condition (ii), and thus $|Z\setminus T|$ is even and $|Z \cap T|$ is even. So by Proposition~\ref{Tjoin}, there exists a $T$-join in $\langle EvenV(G) \rangle_G$. Color all the edges of this $T$-join red, and all the remaining edges blue. Now the resulting red subgraph and  blue subgraph are odd subgraphs and form a partition of $E(G)$.
\end{proof}

Now we prove Theorem~\ref{n&s_condition}.

\medskip \noindent
\textit{Proof of Theorem~\ref{n&s_condition}.}
 Let $\mathcal{X}$ denote the set of components of $\langle OddV(G) \rangle_G$, and let $\mathcal{Y}$ and $\mathcal{Z}$ denote the sets of components of $\langle EvenV(G) \rangle_G$ with odd order and even order, respectively.
 
 Consider the bipartite graph $G^*$, whose vertices correspond to the elements of $\mathcal{X}$ and $\mathcal{Y} \cup \mathcal{Z}$, and an element of $\mathcal{X}$ and that of $\mathcal{Y} \cup \mathcal{Z}$ is joined by an edge if and only if there are odd number of edges of $G$ between the corresponding components. Then it is easy to see that every vertex of $Y\in \mathcal{Y}$
and $Z\in\mathcal{Z}$ has even degree in $G^*$.
 
 Our goal is to give a system of linear equations that is solvable if and only if $G$ is decomposable into two odd subgraphs and its solutions describe partitions satisfying the properties of Theorem \ref{partition-thm}. For every $X_i \in \mathcal{X}$, we assign a binary variable $x_i$ which decides whether $X_i \in \mathcal{R}$ or not. If $x_i =1$, then $X_i \in \mathcal{R}$, and if $x_i = 0$, then $X_i \in \mathcal{B}$.
 Since we want $e_G (R,Y)$ to be odd for every $Y \in \mathcal{Y}$ and $e_G (R,Z)$ to be even for every $Z \in \mathcal{Z}$, consider the following system of linear equations over the binary field $GF(2)=\{0,1\}$.
 \begin{align*}
  \sum_{X_i \in N_{G^*}(Y)} x_i & = 1 \qquad \mbox{for all}~~ Y \in \mathcal{Y} \\
  \sum_{X_i \in N_{G^*}(Z)} x_i & = 0 \qquad \mbox{for all}~~  Z \in \mathcal{Z} 
 \end{align*}

By Theorem \ref{partition-thm}, the multigraph $G$ is decomposable into two odd subgraphs if and only if this system has a solution. The system is solvable if and only if one of the
following three equivalent statements holds. 
  \begin{itemize}
   \item[(i)]  There is no collection of equations such that the sum of the left-hand sides is 0 and the sum of the right-hand sides is 1 (over the binary field).
   \item[(ii)] For any subset of the equations if the sum of the right-hand sides is 1, then there exists a variable $x_i$ which appears odd times in these equations. 
   \item[(iii)] For any $\mathcal{S} \subseteq \mathcal{Y} \cup \mathcal{Z}$ for which $|\mathcal{S} \cap \mathcal{Y}|$ is odd, there exists $X \in \mathcal{X}$ such that $|N_{G^*}(X) \cap \mathcal{S}|$ is odd.
  \end{itemize} 
	Note that statement (iii) is a graph presentation of    statement (ii).
	\hfill \qed

Since a system of linear equations over the binary field can be solved in polynomial time, 
Theorem~\ref{2odd_poly} follows.

However, it is worth translating the algorithm to the language of
graphs.  The steps of the Gauss-elimination can be followed in the
auxiliary bipartite graph $\widehat{G}^*$ which is a slight
modification of the graph $G^*$ used in the proof of
Theorem~\ref{n&s_condition}. In the following we will use ${}^*$ as an
operation that contracts components into single vertices. So the color
classes of $G^*$ are the vertex sets $\mathcal{X}^*$ and
$\mathcal{Y}^*\cup \mathcal{Z}^*$, and our goal is to partition
$\mathcal{X}^*$ into $\mathcal{R}^*$ and $\mathcal{B}^*$. To obtain
$\widehat{G}^*$ a new vertex $b$ is added to $G^*$ and it is connected
to all vertices in $\mathcal{Y}^*$. This vertex $b$ corresponds to the
constant 1 on the right side in the linear equations.

To start the Gauss-elimination we need to select a variable that has a
non-zero coefficient (i.e.~1) in at least two equations and pick one
of these equations. Therefore in $\widehat{G}^*$ we choose an edge
$x_iw$ with $|N_{\widehat{G}^*}(x_i)|\ge 2$, $x_i\in \mathcal{X}^*$
and $w \in \mathcal{Y}^* \cup \mathcal{Z}^*$.  Now in the
Gauss-elimination, we add the equation corresponding to $w$ to all the
equations corresponding to any element of $N_{\widehat{G}^*}(x_i)- \{w \}$ to make the
coefficient of $x_i$ zero in these equations. Then the resulting
system of linear equations corresponds to the bipartite graph
$\widehat{G}^*_1$ that is obtained from $\widehat{G}^*$ by replacing
the induced subgraph
$\langle N_{\widehat{G}^*}(w) \cup \left( N_{\widehat{G}^*}(x_i) - \{
  w \} \right) \rangle_{\widehat{G}^*}$
with its complement.  So $x' \in N_{\widehat{G}^*}(w)$ and
$w'\in N_{\widehat{G}^*}(x_i) - \{ w \}$ are adjacent in
$\widehat{G}^*_1$ if and only if $x'$ and $w'$ are not adjacent in
$\widehat{G}^*$. The other edges are not changed. Notice that the
degree of $x_i$ in $\widehat{G}^*_1$ will be one.

Next we repeat this procedure by choosing an other edge $x_jw'$ in
$\widehat{G}^*_1$ that satisfies the same conditions. Since the degree of
$x_i$ is already one, $x_j$ will automatically differ from the
previously chosen vertices, but we also choose $w'$ to be different
from all previously chosen vertices. If there are no more such edges
then the procedure stops.

Consider the graph of the final stage. At this point we can obtain the
desired partition of the edge set into two odd subgraphs or show the
non-existence of such a partition as follows.

\begin{itemize}
\item If a vertex $w \in\mathcal{Y}^*$ is connected only to the vertex $b$,
then the graph $G$ cannot be decomposed into two odd subgraphs, since
this means that adding up some equations results $0$ on the left-hand
side and $1$ on the right-hand side.  
\end{itemize}
So we may assume that no vertex
$w \in\mathcal{Y}^*$ is connected only to the vertex $b$.  In this
case we obtain a solution as follows.  

\begin{itemize}
\item If a vertex
$x_i \in\mathcal{X}^*$ has degree at least two, then let
$x_i\in \mathcal{B}^*$ and remove all the edges incident with $x_i$. This
means that the variable $x_i$ is a free variable, so it can be set to
0. Thus we may assume that every $x\in \mathcal{X}^*$ is adjacent to
at most one vertex of $\mathcal{Y}^*\cup \mathcal{Z}^* $. Removing these edges makes $x_i$ an isolated vertex, but note that other vertices in $\mathcal{X}^*$ cannot be isolated.

If there is a vertex in $\mathcal{Y}^* \cup \mathcal{Z}^*$ that is adjacent to $b$ and has
more than one neighbors in $\mathcal{X}^*$ (that are all leaves), then let one of these
neighbors be in $\mathcal{R}^*$  and all the others in $\mathcal{B}^*$.
This means that we set one variable to 1 and all the others to 0, so their sum is equal to  1.

\item Otherwise, if $x_i$ is in the same component as $b$, then let $x_i\in \mathcal{R}^*$, meaning that $x_i$ was set to 1 in the solution.
\item If $x_i$ is not in the component of $b$,
then let $x_i \in \mathcal{B}^*$, meaning that $x_i$ was set to 0 in the solution.

\end{itemize}
The above graph operation gives us a partition of $\mathcal{X}^*$ into
$\mathcal{R}^*\cup \mathcal{B}^*$ and the corresponding partition of
$\mathcal{X}$ satisfies the conditions in Theorem~\ref{partition-thm},
and hence $G$ is decomposed into two odd subgraphs.  \qed

\section{Acknowledgment}
The research of the first author was
supported by JSPS KAKENHI Grant Number 16K05248. The research of the second  author was
supported by National Research, Development and Innovation Office NKFIH, K-116769 and K-124171.  The research of the third  author was
supported by National Research, Development and Innovation Office NKFIH, K-124171.


\begin{bibdiv}
\begin{biblist}

\bib{Edmonds1973}{article}{
author={J.~Edmonds},
author={E.L.~Johnson},
title={Matching, Euler tours and the Chinese postman},
journal={Mathematical Programming},
year={1973},
volume={5},
number={1},
pages={88--124},
}

\bib{lovasz}{book}{
 title={Combinatorial problems and exercises},
 author={L.~Lov\'asz},
 date={2007},
 publisher={AMS Chelsea Publishing},
 address={Providence, Rhode Island},
}

\bib{luzar}{article}{
 title={Odd edge coloring of graphs},
 author={B.~Lu\v{z}ar}, author={M.~Petru\v{s}evski}, author={R.~\v{S}krekovski},
 journal={Ars Mathematica Contemporanea},
 volume={9},
 pages={277--287},
 date={2015},
}

\bib{matrai}{article}{
 title={Covering the edges of a graph by three odd subgraphs},
 author={T.~M\'atrai},
 journal={Journal of Graph Theory},
 volume={53},
 pages={75--82},
 date={2006},
}

\bib{petrusevski}{article}{
 title={Odd 4-edge-colorability of graphs},
 author={M.~Petru\v{s}evski},
 journal={Journal of Graph Theory},
 date={2017},
doi = {10.1002/jgt.22168},
}


\bib{pyber}{article}{
 title={Covering the edges of a graph by...},
 author={L.~Pyber},
 journal={Sets, Graphs and Numbers, Colloquia Mathematica Societatis J\'{a}nos Bolyai},
 volume={60},
 pages={583--610},
 date={1991},
}



\end{biblist}
\end{bibdiv}
\end{document}